\documentclass[a4paper,UKenglish,cleveref, autoref, thm-restate]{lipics-v2021}

\title{Rooted Almost-binary Phylogenetic Networks for which the Maximum Covering Subtree Problem is Solvable in Linear Time} 

\titlerunning{Rooted Almost-binary Phylogenetic Networks with a Linear-Time MCSP Algorithm} 
\author{Takatora Suzuki}{Department of Pure and Applied Mathematics, Graduate School of Fundamental Science and Engineering, Waseda University, Japan}{takatora.szk@fuji.waseda.jp}{https://orcid.org/0009-0007-3425-7006}{}
\author{Han Guo}{Department of Pure and Applied Mathematics, Graduate School of Fundamental Science and Engineering, Waseda University, Japan}{}{}{}

\author{Momoko Hayamizu\footnote{Corresponding author}}{Department of Applied Mathematics, Faculty of Science and Engineering, Waseda University, Japan}{hayamizu@waseda.jp}{https://orcid.org/0000-0001-8825-6331}{This work was supported by JST FOREST Program (Grant Number JPMJFR2135, Japan).}

\authorrunning{T. Suzuki, H. Guo and M. Hayamizu} 

\Copyright{Takatora Suzuki, Han Guo and Momoko Hayamizu} 



\category{} 

\relatedversion{} 

\acknowledgements{The authors are grateful to participants at the 6th JST Math Open Problem Workshop for helpful discussions. 
}

\pdfoutput=1 
\nolinenumbers 
\hideLIPIcs  

\EventEditors{A\"{i}da Ouangraoua and Djamal Belazzougui}
\EventNoEds{2}
\EventLongTitle{23rd Workshop on Algorithms in Bioinformatics (WABI 2023)}
\EventShortTitle{WABI 2023}
\EventAcronym{WABI}
\EventYear{2023}
\EventDate{September 3-6, 2023}
\EventLocation{Houston, TX, United States}
\EventLogo{}
\SeriesVolume{}
\ArticleNo{}

\definecolor{mypink}{rgb}{0.9, 0.0, 0.4}
\definecolor{mygreen}{rgb}{0.0,0.5,0.0}
\definecolor{mypurple}{rgb}{0.6,0.0,0.6}
\definecolor{myfuchsia}{rgb}{0.9,0.0,0.9}
\definecolor{mybrown}{rgb}{0.7,0.3,0.3}

\newcommand{\tmp}[1]{\textcolor{blue}{#1}}

\usepackage[subrefformat=parens]{subcaption}
\newtheorem{problem}[theorem]{Problem}

\begin{document}

\maketitle

\begin{abstract}
Phylogenetic networks are a flexible model of evolution that can represent reticulate evolution and deal with complex biological data. Tree-based networks, which are phylogenetic networks that contain a spanning tree with the same root and leaf-set as the network itself, have been intensively studied. Although tree-based networks are mathematically tractable in many ways, not all networks are tree-based. Francis--Semple--Steel (2018) thus introduced some different indices for measuring the deviation of rooted binary phylogenetic networks $N$ from being tree-based, such as the minimum number $\delta^\ast(N)$ of additional leaves needed to make $N$ tree-based, and the minimum difference $\eta^\ast(N)$ between the number of vertices of $N$ and the number of vertices of a subtree of $N$ that shares the root and leaf set with $N$. Hayamizu~(2021) has established a canonical decomposition of rooted binary (and almost-binary) phylogenetic networks of $N$, called the maximal zig-zag trail decomposition, which has many implications including a linear time  algorithm for computing $\delta^\ast(N)$. The Maximum Covering Subtree Problem (MCSP)  is the problem of computing $\eta^\ast(N)$ and Davidov \textit{et al.}~(2022) showed that this  can be solved in polynomial time (in cubic time when $N$ is binary) by an algorithm for the minimum cost flow problem. 
In this paper, under the assumption that $N$ is \emph{almost-binary} (i.e.  each internal vertex has in-degree and out-degree at most two), we show that $\delta^\ast(N)\leq \eta^\ast (N)$ holds, which is tight, and give a characterisation of such phylogenetic networks $N$ that satisfy $\delta^\ast(N)=\eta^\ast(N)$. Our approach is based on the canonical decomposition of $N$ and focuses on how the maximal W-fences (i.e. the forbidden subgraphs of tree-based networks) are connected with maximal M-fences in the network $N$. Our results introduce a new class of phylogenetic networks for which MCSP can be solved in linear time, and such networks can be viewed a generalisation of tree-based networks. 
\end{abstract}

\section{Introduction}
\label{sec:intro}
Phylogenetic trees are a standard model of evolution, which represent branching histories of species. Because of their simplicity, trees are mathematically and computationally tractable, and many methods have been developed and implemented for various computational problems. However, the process of evolution cannot be explained by branching alone when reticulation events such as horizontal gene transfer and hybridisation occur. It is therefore important to explore a more general model beyond trees in order to more accurately describe complex evolution.

Phylogenetic networks provide a more flexible model of evolution than phylogenetic trees because they allow  lineages to merge and diverge. Phylogenetic networks can be used not only in situations such as hybridisation, where two lineages merge, but also to deal with complex information, such as inconsistencies between trees inferred from different genes. If it is possible to integrate multiple trees together into a single phylogenetic network by some method, the interpretation of evolutionary scenarios can be made easier by searching for a single probable tree contained in that network. This idea can be seen, for example,  in the notion of `statistical tree of life' by O’Malley and Koonin (2011) \cite{o2011stands}. 

Although there are several ways to define phylogenetic trees within a phylogenetic network, there has been much discussion about what might be called `phylogenetic spanning trees' (also known as `subdivision trees' or `support trees') of phylogenetic networks \cite{which_phylogenetic_networks,anaya2016determining,on_tree_baesd_phylogenetic, structure_theorem,ranking_top_k,universal_tree_based,pons2019tree,van2021unifying,jetten2016nonbinary}. Following this line of research, Hayamizu (2021)~\cite{structure_theorem} established a canonical way to decompose any rooted binary phylogenetic network (and, more generally, any almost-binary network) into a unique set of subgraphs called maximal zig-zag trails. In \cite{structure_theorem}, Hayamizu used it to obtain the structure theorem for phylogenetic networks and provided linear-time and linear-delay algorithms for many computational problems related to subdivision trees, such as counting, listing, and optimisation. In addition, Hayamizu and Makino (2022) \cite{ranking_top_k} gave a linear-delay algorithm for the more general problem of generating subdivision trees sequentially from an optimal one to a $k$-th optimal one, according to the values of a certain type of objective function, such as likelihood or log-likelihood. Thus, the structure theorem has led to the development of such fast algorithms, and this has been made possible because it explicitly characterises the set of all phylogenetic spanning trees of a given network $N$. We note that it is not possible to obtain similar algorithms for the set of all phylogenetic subtrees of $N$ (i.e. the so-called `displayed' trees of $N$) rather than spanning trees. In fact, the problem of just counting the subtrees displayed by $N$ is known to be \#P-complete~\cite{counting_trees} (see e.g. \cite{kanj2008seeing, locating_a_tree} for other hardness results related to displayed trees). 

Phylogenetic networks with a phylogenetic spanning tree are called tree-based networks, but not all networks are tree-based. Francis--Semple--Steel (2018)~\cite{New_characterisations} introduced several indices to measure how far a given network $N$ deviates from being tree-based, such as the minimum number $\delta^\ast(N)$ of extra leaves that need to be added to make $N$ tree-based, and the minimum difference $\eta^\ast(N)$ between the number of vertices of $N$ and the number of vertices of a subtree of $N$ with the same root and the same set of leaves as $N$. Hayamizu \cite{structure_theorem} proved $\delta^\ast(N)=|\mathcal{W}_N|$, where $|\mathcal{W}_N|$ is the number of maximal W-fences of $N$, and provided a linear-time algorithm for computing $\delta^\ast(N)$. The problem of computing $\eta^\ast(N)$ is called the Maximum Covering Subtree Problem (MCSP), and Davidov et~al.~\cite{Maximum_Covering} showed that MCSP can be solved in polynomial time even for non-binary $N$ (cubic time if $N$ is binary) by an algorithm for solving the minimum cost flow problem.

In this paper, we study the number $\eta^\ast(N)$ of vertices of $N$ that are not covered by a maximum covering subtree, using a different approach from the one used in Davidov~\textit{et al.}~\cite{Maximum_Covering}. We use the above-mentioned canonical decomposition of $N$ into the set of unique maximal zig-zag trails, where $N$ is assumed to be binary or almost-binary. Focusing on the number $|\mathcal{W}_N|$ of maximal W-fences in $N$, we prove that $|\mathcal{W}_N|$ gives a lower bound on $\eta^\ast(N)$, i.e. $\eta^\ast(N) \geq |\mathcal{W}_N|=\delta^\ast (N)$. This lower bound is tight, i.e. there exist phylogenetic networks $N$ for which $\eta^\ast(N) = |\mathcal{W}_N|$ (e.g. tree-based networks). We provide a necessary and sufficient condition for $N$ to satisfy $\eta^\ast(N) = |\mathcal{W}_N|$, thus introducing a class of phylogenetic networks for which MCSP can be solved in linear time.  Such phylogenetic networks $N$ can be seen as a generalisation of tree-based networks.

\section{Preliminaries}\label{sec:prelim}
\subsection{Graph theoretical terminology}
In this paper, we consider only (weakly) connected, finite, simple, directed acyclic graphs which we now define. 
A \emph{directed graph} is defined as an ordered pair $(V, E)$ of a set $V$ of vertices and a set $E$ of directed edges, whereas an \emph{undirected graph} consists of a set of vertices and a set of undirected edges. An undirected graph is \emph{connected} if there exists a path between any pair of vertices, and a directed graph $G$ is \emph{(weakly) connected} if the undirected graph obtained by ignoring the direction of each edge of $G$ is connected. From now on, we will use the terms `graph' and `edge' to mean a directed graph and directed edge, respectively.  
Given a graph $G$, its vertex-set and edge-set are denoted by $V(G)$ and $E(G)$, respectively. A graph $G$ is said to be \emph{finite} if $V(G)$ and $E(G)$ are finite sets. An edge oriented from vertex $u$ to vertex $v$ is denoted by $(u, v)$. 
A graph is said to be \emph{simple} if it does not contain any loops (i.e.  edges that start and end at the same vertex) or multiple edges (i.e.  multiple directed edges between the same pair of vertices).

For an edge $(u, v)$  of a directed graph with $u\neq v$, we say that $u$ is a \emph{parent} of $v$ and $v$ a \emph{child} of $u$.  Given an edge $e=(u, v)$, we often use the notation $tail(e)$ and $head(e)$ to mean the starting point $u$ and the ending point $v$, respectively. 
For a vertex $v$ of a graph $G$, the \emph{in-degree of $v$ (in $G$)}, denoted by $indeg_G(v)$,  is defined to be the number of edges $e$ of $G$ with $head(e)=v$. Similarly, the \emph{out-degree of $v$ (in $G$)}, denoted by $outdeg_G(v)$,  is defined to be the number of edges $e$ of $G$ with $tail(e)=v$.



A \emph{(directed) path} is a directed graph $G$ that can be represented by an alternating sequence of vertices and consecutive edges $v_1, (v_1, v_2), v_2, \dots, (v_{k-1}, v_k), v_k$, where all vertices are distinct and we have $(indeg_G(v_1), outdeg_G(v_1))=(0,1)$, $(indeg_G(v_k), outdeg_G(v_k))=(1, 0)$ and $(indeg_G(v_i), outdeg_G(v_i))=(1,1)$ for any vertex $v_i$ other than $v_1, v_k$. 
A \emph{(directed) cycle} is a directed graph $G$ that can be represented by an alternating sequence of vertices and consecutive edges $v_1, (v_1, v_2), v_2, \dots, (v_{k-1}, v_k), v_k$, where all vertices are distinct, $v_1=v_k$, and $(indeg_G(v_i), outdeg_G(v_i))=(1,1)$ for any vertex $v_i$ of $G$.
For two graphs $G$ and $H$, $H$ is called a \emph{subgraph} of $G$ if $V(H) \subseteq V(G)$ and $E(H) \subseteq E(G)$. 
A graph is \emph{acyclic} if it does not contain any directed cycles.  
A \emph{(directed) tree} is a directed acyclic graph $G$ such that there exists a unique vertex $\rho$ of $G$ with $indeg_G(\rho)=0$ and for any vertex $v$ of $G$ other than $\rho$, we have $indeg_G(v)=1$. 
A subgraph $H$ of a graph $G$ is called a \emph{subtree} of $G$ if $H$ is a directed tree. A subtree $H$ of $G$ with $V(H)=V(G)$ is called a \emph{spanning tree} of $G$. 
For a vertex $v$ of a graph $G$, \emph{removing $v$} means deleting $v$ and all the edges starting or ending at $v$, and we write $G- \{v\}$ to represent the subgraph of $G$ obtained by removing $v$ from $G$.

For a graph $G = (V, E)$ and any subset $E^\prime \subset E$, $E^\prime$ is said to \emph{induce} the subgraph   $(V^\prime, E^\prime)$ of $G$, 
where $V^\prime:=\{ v\in V \mid v=head(e) \text{ or } v=tail(e) \text{ for some } e \in E^\prime  \}$. For notational simplicity, we write $G[E^\prime]$ to represent the edge-induced subgraph $(V^\prime, E^\prime)$ of $G$. For a graph $G = (V, E)$ and any partition $\{E_1, \dots ,E_n\}$ of $E$, the collection $\{ G[E_1],\dots, G[E_n]\}$ of edge-induced subgraphs  is called a \emph{decomposition} of $G$.

Two directed graphs $G$ and $H$ are said to be \emph{isomorphic} if there exists a bijection between $V(G)$ and $V(H)$ that yields a bijection between $E(G)$ and $E(H)$.  
In other words, $G$ and $H$ are isomorphic if there exists a labelling map $f: V(G)\to V(H)$ so that there is a one-to-one correspondence between $V(G)$ and $V(H)$ and for any $v_i$ and $v_j$ in $V(G)$, $(v_i, v_j)$ is in $E(G)$ if and only if  $(f(v_i),  f(v_j))$ is in $E(H)$, including the loops and parallel edges if they exist. 
For an edge $(u, v)$ of a graph $G$,  \emph{subdividing edge $(u, v)$} refers to the operation of introducing a new vertex $p$ and replacing $(u, v)$ with the new consecutive edges $(u, p)$ and $(p, v)$. Any graph that can be obtained by subdividing each edge of $G$ zero or more times is called a \emph{subdivision} of $G$. Conversely, for a vertex $p$ of a graph $G$ with $indeg_G(p) = outdeg_G(p) = 1$ and two consecutive edges $(u, p)$ and $(p, v)$, \emph{smoothing vertex $p$} refers to the operation of removing the vertex $p$ together with the edges $(u, p)$ and $(p, v)$ and then adding a new edge $(u, v)$. Two graphs are said to be \emph{homeomorphic} if they are isomorphic after smoothing all vertices of in-degree one and out-degree one.

\subsection{Phylogenetic networks}
Throughout this paper, $X=\{1,\dots, n\}$ represents a non-empty finite set, which can be biologically interpreted as a set of $n$ present-day species.  
In this paper, we slightly generalise the notion of `binary' phylogenetic $X$-networks and consider `almost-binary' phylogenetic $X$-networks. This is because, as discussed in~\cite{structure_theorem}, the structure theorem for rooted binary phylogenetic $X$-networks can be applied not only to binary but also to almost-binary $X$-networks.

\begin{definition}[See also Figure~\ref{fig:Tree,Net}]\label{phyloNet}
  A \emph{rooted almost-binary phylogenetic $X$-network} is defined to be any simple directed acyclic graph $N = (V, E)$ with the following properties:
  \begin{enumerate}
    \item There exists a unique vertex $\rho \in V$ with $indeg(\rho) = 0, outdeg(\rho) \in \{ 1, 2 \}$;
    \item There exists a bijection between $X$ and $L:=\{v\in V \mid indeg(v) = 1, outdeg(v) = 0\}$;
    \item  For any $v \in V \setminus (X \cup \{\rho\})$, $indeg(v)\in \{1, 2\}$ and $outdeg(v)\in \{1, 2\}$ hold.
  \end{enumerate} 
\end{definition}

According to Definition \ref{phyloNet}, the vertex set $V$ of a rooted almost-binary phylogenetic $X$-network consists of three types of vertices: a unique vertex $\rho$ called \emph{the root} of $N$, which represents the most recent common ancestor of the species in the set $X$; the vertices in $L$, called \emph{leaves} of $N$, that have no descendants in $N$ and thus can be identified with the present-day species in $X$; and the remaining non-root, non-leaf vertices, 
 which have one or two incoming edges and one or two outgoing edges. 



As illustrated Figure~\ref{fig:Tree,Net}, a rooted almost-binary phylogenetic $X$-network $N$ is called a rooted \emph{binary} phylogenetic $X$-network if $N$ contains neither a vertex $v$ with $indeg_N(v) = outdeg_N(v) = 2$ nor a vertex $v$ with $indeg_N(v) = outdeg_N(v) = 1$. Moreover, if $N$ does not also contain a vertex $v$ with $(indeg_G(v), outdeg_N(v))=(2, 1)$, then $N$ is called a rooted binary phylogenetic $X$-\emph{tree}.
 A vertex $v$ with $(indeg_N(v), outdeg_N(v)) = (1, 2)$ is called a \emph{tree vertex} of $N$, whereas a vertex $v$ with $(indeg_N(v), outdeg_N(v)) = (2, 1)$ is called a \emph{reticulation} of $N$.  Intuitively, a tree vertex is a branching point in an evolutionary history, whereas a reticulation can be viewed as a point where two different lineages merge. 

\begin{figure}[hbt]
  \centering
  \includegraphics[scale=0.4]{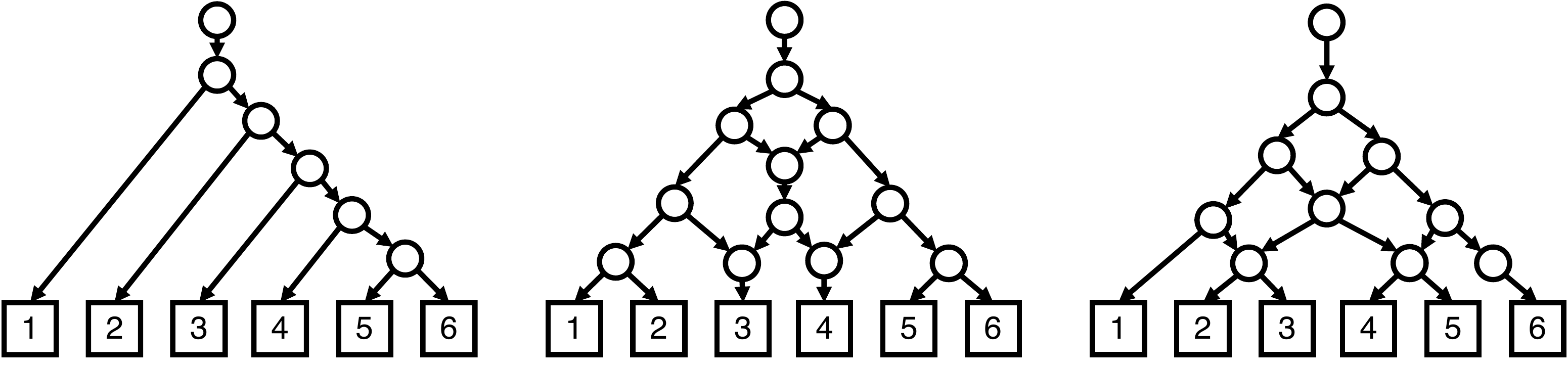}
  \caption{Examples of a rooted binary phylogenetic $X$-tree (left), a rooted binary phylogenetic $X$-network (center), a rooted almost-binaty phylogenetic $X$-network (right) where $X = \{1, \dots, 6\}$.
  \label{fig:Tree,Net}
  }
\end{figure}

\section{Maximum Covering Subtree Problem and Known Results}
In this section, we formally describe the Maximum Covering Subtree Problem (MCSP), which was originally posed by Francis--Semple--Steel~\cite{New_characterisations}, and summarise relevant known results from the work by Davidov~\textit{et al.}~\cite{Maximum_Covering}. 

We first recall the definition of tree-based networks and subdivision trees, as maximum covering subtrees are closely related to these concepts. A rooted almost-binary phylogenetic $X$-network $N$ is called a \emph{tree-based network (on $X$)} if $N$ has a spanning tree (i.e.  a subtree $T$ with $V(T)=V(N)$) that is a subdivision of some rooted binary phylogenetic $X$-tree, and in this case $T$ is called a \emph{subdivision tree} of $N$ (see Figure~\ref{fig:tree-based} for an illustration). Note that this definition implies that a subdivision tree $T$ of $N$ not only contains all vertices of $N$, but also shares the root and leaf-set $X$ with $N$. 

\begin{figure}[hbt]
  \centering
  \includegraphics[scale=0.4]{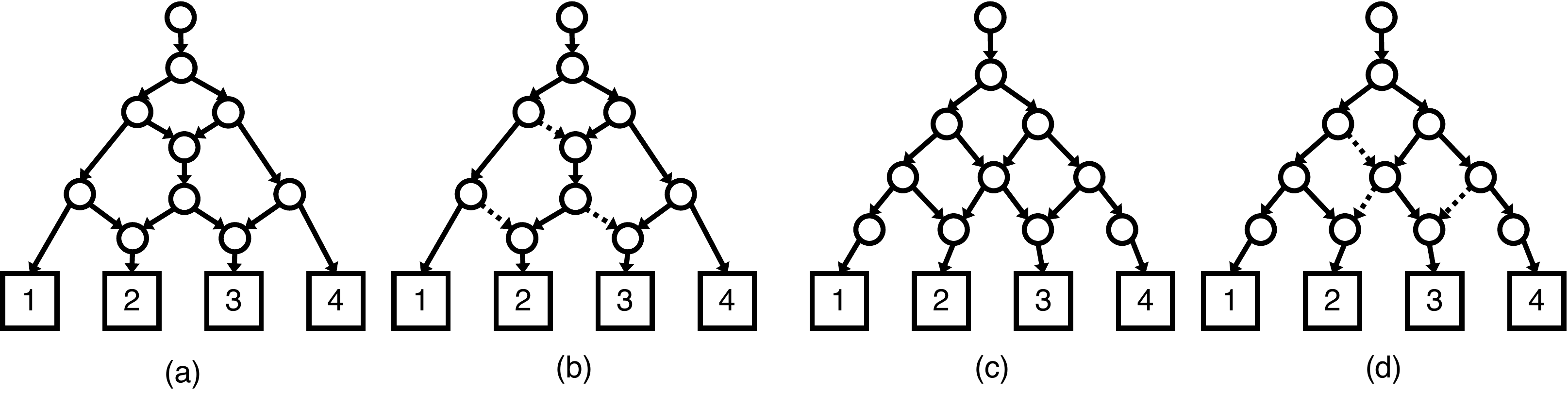}
  \caption{Examples of rooted binary/almost-binary tree-based networks and subdivion trees. A rooted binary phylogenetic $X$-network in (a) has a subdivision tree in (b), therefore the network in (a) is tree-based. Likewise, a rooted almost-binary phylogenetic $X$-network in (c) has a subdivision tree in (d), therefore the network in (c) is tree-based.
  \label{fig:tree-based}
  }
\end{figure}


A \emph{covering subtree} of a rooted (not necessarily binary) phylogenetic $X$-network $N$ is defined to be a subtree whose root and leaf-set is the same as those of $N$. A covering subtree of $N$ with the largest number of vertices is called a \emph{maximum} covering subtree of $N$. We now state the problem as follows. See also Figure~\ref{fig:MCST} for an illustration.



\begin{problem}[Maximum Covering Subtree Problem]\label{prob}
  Given a rooted (not necessarily binary) phylogenetic $X$-network $N$,  compute the number $\eta^\ast(N):=|V(N)|-|V(T)|$, where $T$ denotes a maximum covering subtree of $N$. 
\end{problem}

\begin{figure}[h]
  \centering
  \includegraphics[scale=0.4]{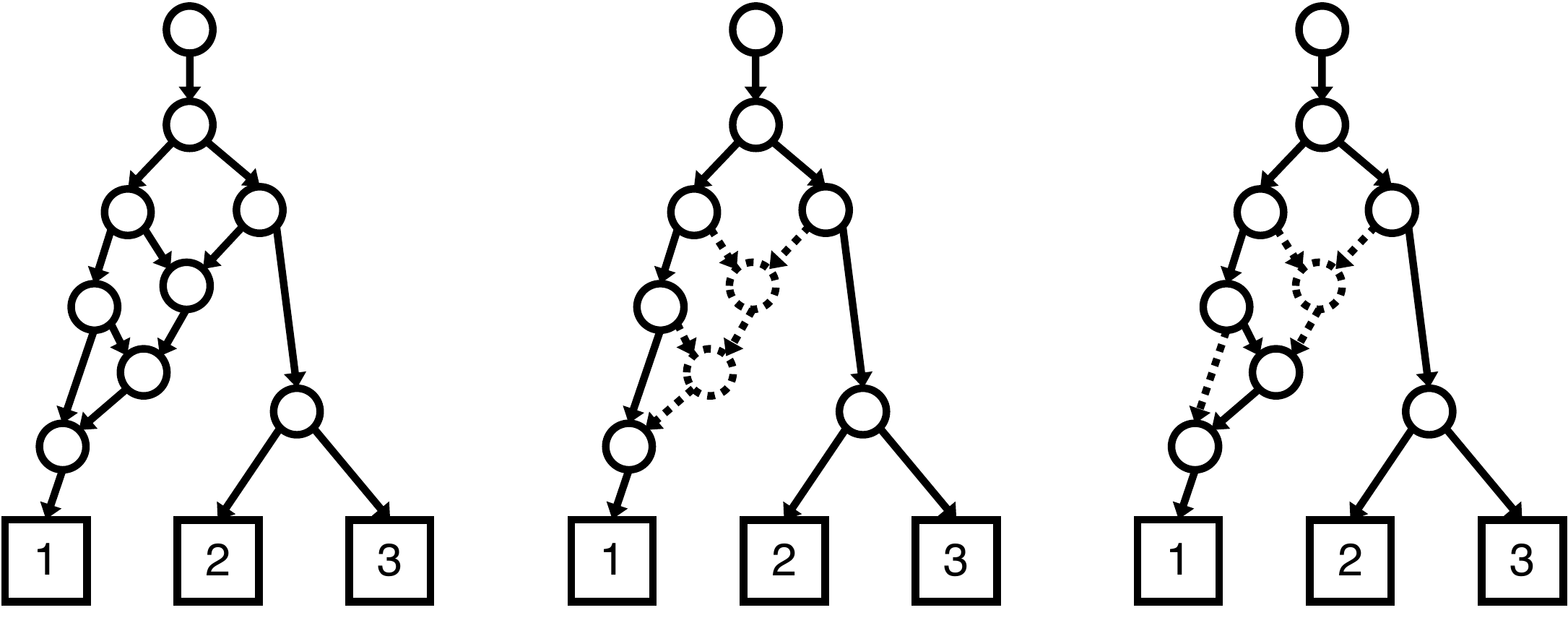}
  \caption{An illustration of Problem~\ref{prob}. Left: A rooted almost-binary phylogenetic $X$-network $N$. Middle: A covering subtree with 10 vertices. Right: A maximum covering subtree $T$ with 11 vertices. In this case, $\eta^\ast(N):=|V(N)|-|V(T)|=1$. 
  \label{fig:MCST}}
\end{figure}

Problem~\ref{prob} asks for the number $\eta^*(N)$ of vertices of $N$ that are not covered by a maximum covering subtree of $N$.
We note that $\eta^*(N)=0$ holds if and only if $N$ is a tree-based network and thus $\eta^*(N)$ can be interpreted as quantifying how far $N$ deviates from being tree-based.
Davidov~\textit{et al.}~\cite{Maximum_Covering} provided an algorithm that solves Problem~\ref{prob} (and also finds a maximum covering subtree $T$) in $O(|V|^2 |E|)$ time. Under the assumption that the input network $N$ is almost-binary, the algorithm runs in $O(|V|^3)=O(|E|^3)$ time.

\section{Canonical Decomposition of Rooted Almost-Binary Phylogenetic $X$-Networks and its Implications for Tree-based Networks}
We recall the relevant definitions and results that will be used in this paper from \cite{structure_theorem} without providing proofs, thus ensuring our exposition is self-contained. Although the results in \cite{structure_theorem} were originally intended for rooted binary phylogenetic $X$-networks, all of them apply to rooted almost-binary phylogenetic $X$-networks as well, as pointed out in Section 6 of \cite{structure_theorem}.

\begin{definition}
	For a rooted almost-binary phylogenetic $X$-network $N$, a \emph{zig-zag trail in $N$} is defined to be a connected subgraph $Z$ of $N$ with $|E(Z)|\geq 1$, such that there is a labelling $(e_1,\dots, e_m)$ of the edges of $Z$ where any adjacent edges $e_i$ and $e_{i+1}$ of $Z$ satisfy either ${\it head}(e_i)={\it head}(e_{i+1})$ or ${\it tail}(e_i)={\it tail}(e_{i+1})$. A zig-zag trail $Z$ in $N$ is said to be \emph{maximal} if $N$ contains no zig-zag trail $Z^\prime$ such that $Z$ is a proper subgraph of $Z^\prime$. Note that a maximal zig-zag trail may consist of a single edge. 
\end{definition}

Any zig-zag trail $Z$ in a rooted binary network $N$ can be denoted by an alternating sequence of (not necessarily distinct) vertices and distinct edges. However, if we drop the edges from the notation, we can more concisely express $Z$ as  $v_1>v_2<v_3>v_4< \cdots > v_{m-1} < v_m$ (or in reverse order, $v_m>v_{m-1}< \cdots > v_4 < v_3 > v_2 < v_1$). We sometimes even write $(e_1,\dots, e_m)$ when there is no confusion without specifying the edge directions. Every maximal zig-zag trail falls into one of the four types defined in Definition~\ref{def:zig-zags.4types} (see also  Figure~\ref{fig:zig-zag-trail}). A similar definition is also found in  \cite{on_tree_baesd_phylogenetic}. 


\begin{definition}\label{def:zig-zags.4types}
Let $N$ be a rooted almost-binary phylogenetic $X$-network and let $Z$ be a maximal zig-zag trail in $N$. The four types of maximal zig-zag trails is defined as follows:
\begin{itemize}
	\item $Z$ is  a \emph{crown} if $Z$ has an even number $|E(Z)|\geq 4$ of edges and $Z$ can be represented as \\ $v_1 < v_2 > v_3 < \cdots > v_{2k-1} < v_{2k} > v_{2k+1}=v_1$. 
	\item  $Z$ is a \emph{M-fence} if $Z$ has an even number $|E(Z)|\geq 2$ of edges and $Z$ can be represented as \\ $v_1 < v_2 > v_3 < \cdots > v_{2k-1} < v_{2k} > v_{2k+1}\neq v_1$.  
	\item $Z$ is  an \emph{N-fence} if $Z$ has an odd number $|E(Z)|\geq 1$ of edges and $Z$ can be represented as \\ $v_1 < v_2 > v_3 < \cdots > v_{2k-1} < v_{2k}$.
	\item  $Z$ is a \emph{W-fence} if $Z$ has an even number $|E(Z)|\geq 2$ of edges and $Z$ can be represented as \\ $v_1 > v_2 < v_3 > \cdots < v_{2k-1} > v_{2k} < v_{2k+1}$.  
\end{itemize}
\end{definition}

\begin{figure}[hbt]
  \centering
  \includegraphics[scale=0.5]{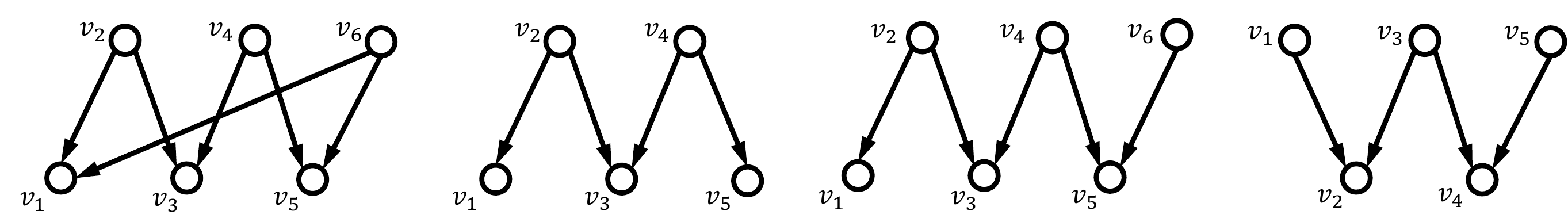}
  \caption{An illustration of the four types of maximal zig-zag trails in rooted almost-binary phylogenetic $X$-networks in Definition~\ref{def:zig-zags.4types} (from left to right: a crown, an M-fence, an N-fence, and a W-fence). 
   \label{fig:zig-zag-trail}}
\end{figure}

\begin{remark}\label{rem:atypical.fence}
	We note that fences do not necessarily have to have the shapes of the letters M, N, and W as described in Figure~\ref{fig:zig-zag-trail}. For instance, see the M-fence $Z_3$ in the network on the right of Figure~\ref{fig:zig-zag-trail-decomposition}.
\end{remark}

In this paper, our entire argument will be based on the following Theorem~\ref{uniquely.decomposable}, 
which provides a canonical way of decomposing any rooted almost-binary phylogenetic $X$-network into a unique set of maximal zig-zag trails. Theorem~\ref{uniquely.decomposable} was originally proved for binary networks in \cite{structure_theorem}, but as was noted in Section 6 of \cite{structure_theorem}, all the arguments, statements and algorithms in \cite{structure_theorem} hold true for any almost-binary networks. Therefore, we here do not repeat the proof of Theorem~\ref{uniquely.decomposable} although it is short and straightforward. The interested reader may refer to the proof of Theorem~4.2 in \cite{structure_theorem}.


 \begin{theorem}[\cite{structure_theorem}]\label{uniquely.decomposable}
 For any  rooted almost-binary phylogenetic $X$-network $N$,  there exists a unique decomposition $\mathcal{Z}=\{Z_1,\dots,Z_k\}$  of $N$ such that each $Z_i\in \mathcal{Z}$ is a maximal zig-zag trail in $N$ (see Figure~\ref{fig:zig-zag-trail-decomposition} for an illustration). 
\end{theorem}

The unique decomposition of $\mathcal{Z}$ of $N$ in Theorem~\ref{uniquely.decomposable} is called \emph{the maximal zig-zag trail decomposition}  of $N$. Using such a canonical decomposition, Hayamizu \cite{structure_theorem} provided an explicit characterisation of the set of all subdivision trees (spanning trees with leaf-set $X$) in a given tree-based phylogenetic network $N$ on $X$ as a direct product of families of sets of possible edge-choices within each subgraph $Z_i$ of $N$. 

Such a unique decomposition and characterisation in direct product form is in the spirit of various structure theorems in mathematics, such as the structure theorem for finite Abelian groups, which states that every finite Abelian group can be uniquely decomposed as a direct product of cyclic groups. In light of this, when we speak of the ``structure theorem'' for rooted binary (or almost-binary) phylogenetic networks, it would be appropriate to refer to both the uniqueness of the decomposition (Theorem~\ref{uniquely.decomposable}) and the direct-product  characterisation of the set of all subdivision trees (Theorem~4.8 in \cite{structure_theorem}). 
For convenience, however, we often use the name ``structure theorem'' to refer only to Theorem~\ref{uniquely.decomposable} (and this was the case even in the original paper~\cite{structure_theorem}).  In this paper we  only need Theorem~\ref{uniquely.decomposable}, so we do not restate the full statement of the structure theorem here.  For readers interested in the rest of the structure theorem, we refer to Theorem~4.8 in \cite{structure_theorem}.

\begin{figure}[hbt]
  \centering
\includegraphics[scale=0.5]{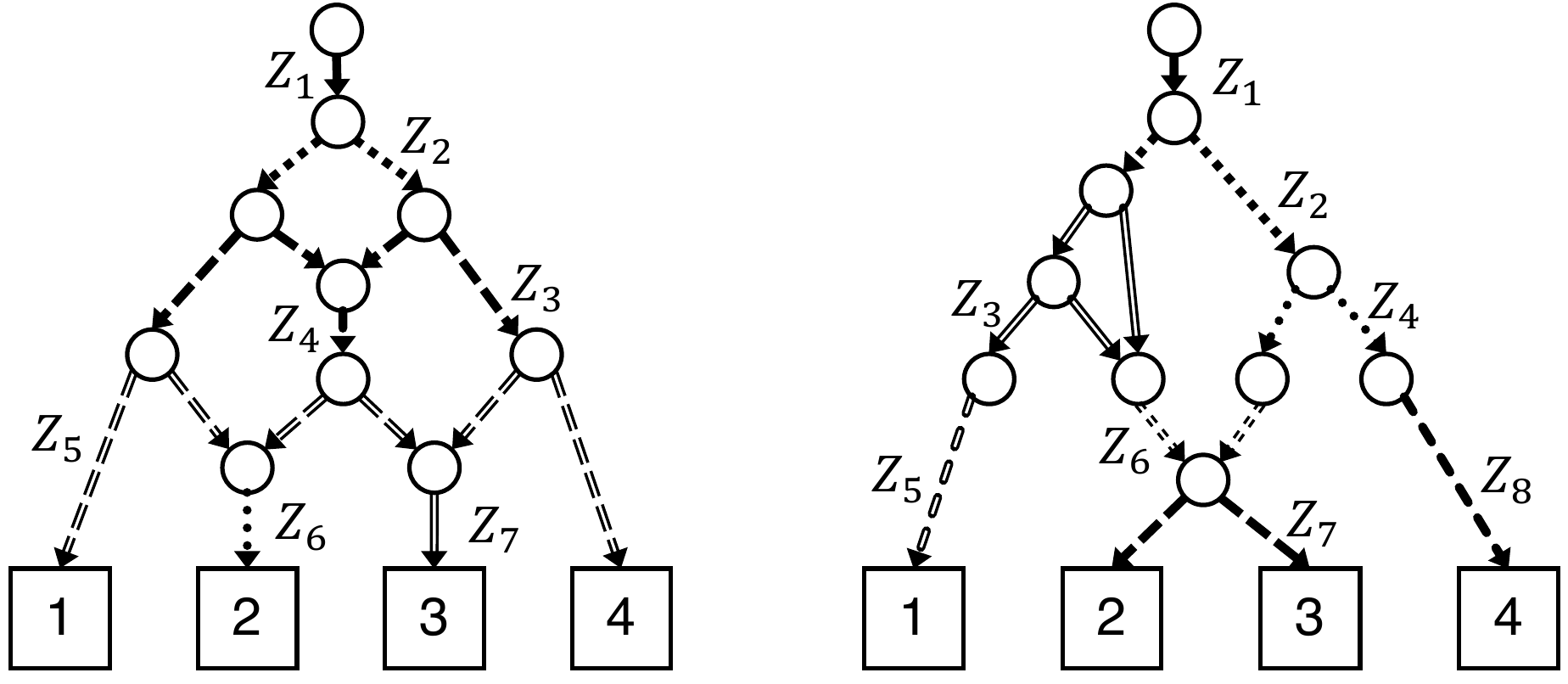}
  \caption{An illustration of the maximal zig-zag trail decomposition of a rooted binary tree-based $X$-network (left) and a rooted almost-binary phylogenetic $X$-network (right). 
  \label{fig:zig-zag-trail-decomposition}
  }
\end{figure}

%
%
%


The next Proposition~\ref{spanning} has been stated and proved in independent studies using different proofs and terminology. 

\begin{proposition}[(Zhang (2016) \cite{on_tree_baesd_phylogenetic}, Jetten and van Iersel (2016) \cite{jetten2016nonbinary}, Hayamizu (2021) \cite{structure_theorem})]\label{spanning}
	Let $N$ be a rooted binary phylogenetic $X$-network  and  $\mathcal{Z}=\{Z_1,\dots,Z_k\}$  be the maximal zig-zag trail decomposition of $N$.  Then, $N$ is a tree-based network on $X$ if and only if no element $Z_i\in \mathcal{Z}$ is a W-fence.
\end{proposition}

Proposition~\ref{spanning} is just one of many interesting consequences of Theorem~\ref{uniquely.decomposable} (see \cite{structure_theorem, ranking_top_k} for many other results and various linear-time algorithms  derived from Theorem~\ref{uniquely.decomposable}) but is worth recalling  before considering the Maximum Covering Subtree Problem (Problem~\ref{prob}). This is because Proposition~\ref{spanning} means that  $\eta^\ast(N)=0$ holds if and only if the number $|\mathcal{W}_N|$ of maximal W-fences in $N$ is exactly zero. This implies that if $N$ is a tree-based network, then we have $|\mathcal{W}_N|=\eta^\ast(N)=0$. 

Since the number $|\mathcal{W}_N|$ of maximal W-fences in $N$ can be computed in linear time \cite{structure_theorem}, a natural question arises as to when $|\mathcal{W}_N|=\eta^\ast(N)$ holds.  
In Section~\ref{sec:results}, we will provide a necessary and sufficient condition for when $|\mathcal{W}_N|=\eta^\ast(N)$ holds  (Theorem~\ref{equality}). 

\section{Results: New Approach for MCSP using Canonical Decomposition of Phylogenetic Networks}\label{sec:results}

In this section, we will use the notation $\mathcal{W}_N$ and $\mathcal{M}_N$ to represent the sets of maximal W-fences and maximal M-fences in $N$.  Also, since we will only consider maximal zig-zag trails, when no confusion arises, the word ``maximal'' may be omitted from the description of crowns, M-fences, N-fences and W-fences in $N$.
In Section~\ref{subsec:ineq}, we  prove $|\mathcal{W}_N|\leq \eta^\ast(N)$  (Theorem~\ref{infimum_of_eta}).  In  Section~\ref{subsec:eq}, we provide a necessary and sufficient condition for when $|\mathcal{W}_N|=\eta^\ast(N)$ holds (Theorem~\ref{equality}).

%

\subsection{Bounding $\eta^\ast(N)$ from below by the number $|\mathcal{W}_N|$ of maximal W-fences}\label{subsec:ineq}

As mentioned earlier in Remark~\ref{rem:atypical.fence}, a zig-zag trail is not necessarily a bipartite graph as depicted in Figure~\ref{fig:zig-zag-trail}. In the case of $Z_3$ on the right of Figure~\ref{fig:zig-zag-trail-decomposition}, for example, there is no obvious vertical relationship between the vertices of $Z_3$. However, since any maximal zig-zag trail $Z_i$ can be expressed as in  Definition~\ref{def:zig-zags.4types}, 
we can consider `upper' and `lower' vertices of $Z_i$ that are formally defined in Definition~\ref{upper,lower}.

%


\begin{definition}\label{upper,lower}
Let $Z$ be a maximal zig-zag trail in a rooted almost-binary phylogenetic network $N$ as described in Definition~\ref{def:zig-zags.4types}. 
If $Z$ is a W-fence, we call each odd-labelled vertex $v_{2i-1}$ (resp. even-labelled vertex  $v_{2i}$) of $Z$ an \emph{upper vertex} (resp. \emph{lower vertex}) of $Z$. Otherwise, we call each even-labelled vertex  $v_{2i}$ (resp. odd-labelled vertex  $v_{2i-1}$) of $Z$ an \emph{upper vertex} (resp. \emph{lower vertex}) of $Z$. 
The set of upper vertices and the set of lower vertices of $Z$ is  denoted by $V_{u}(Z)$ and by $V_{\ell}(Z)$, respectively.
\end{definition}

By Definition \ref{upper,lower}, each upper vertex $v$ of a maximal zig-zag trail $Z$ satisfies  $outdeg_{Z} (v) \geq 1$, and each lower vertex $v$ of $Z$ satisfies  $indeg_{Z} (v) \geq 1$. Therefore, an upper vertex $v$ of a maximal zig-zag trail $Z$ can be a lower vertex of $Z$ at the same time (see the above-mentioned $Z_3$ in Figure~\ref{fig:zig-zag-trail-decomposition}). 
The next proposition follows from Definition~\ref{def:zig-zags.4types}. 
\begin{proposition}\label{four_types_of_zig-zag_trail}
 Let $Z$ be a maximal zig-zag trail in a rooted almost-binary phylogenetic network $N$. Then, 
\begin{align}\label{eq:diff}
  |V_\ell(Z)| - |V_u (Z)| = \begin{cases}
    1 &\text{if $Z$ is an M-fence;}\\
    0 &\text{if $Z$ is a crown or an N-fence;}\\
    -1 &\text{if $Z$ is a W-fence.}
  \end{cases}
\end{align}
%

\end{proposition}

\begin{figure}[hbt]
  \centering
  \includegraphics[scale=0.6]{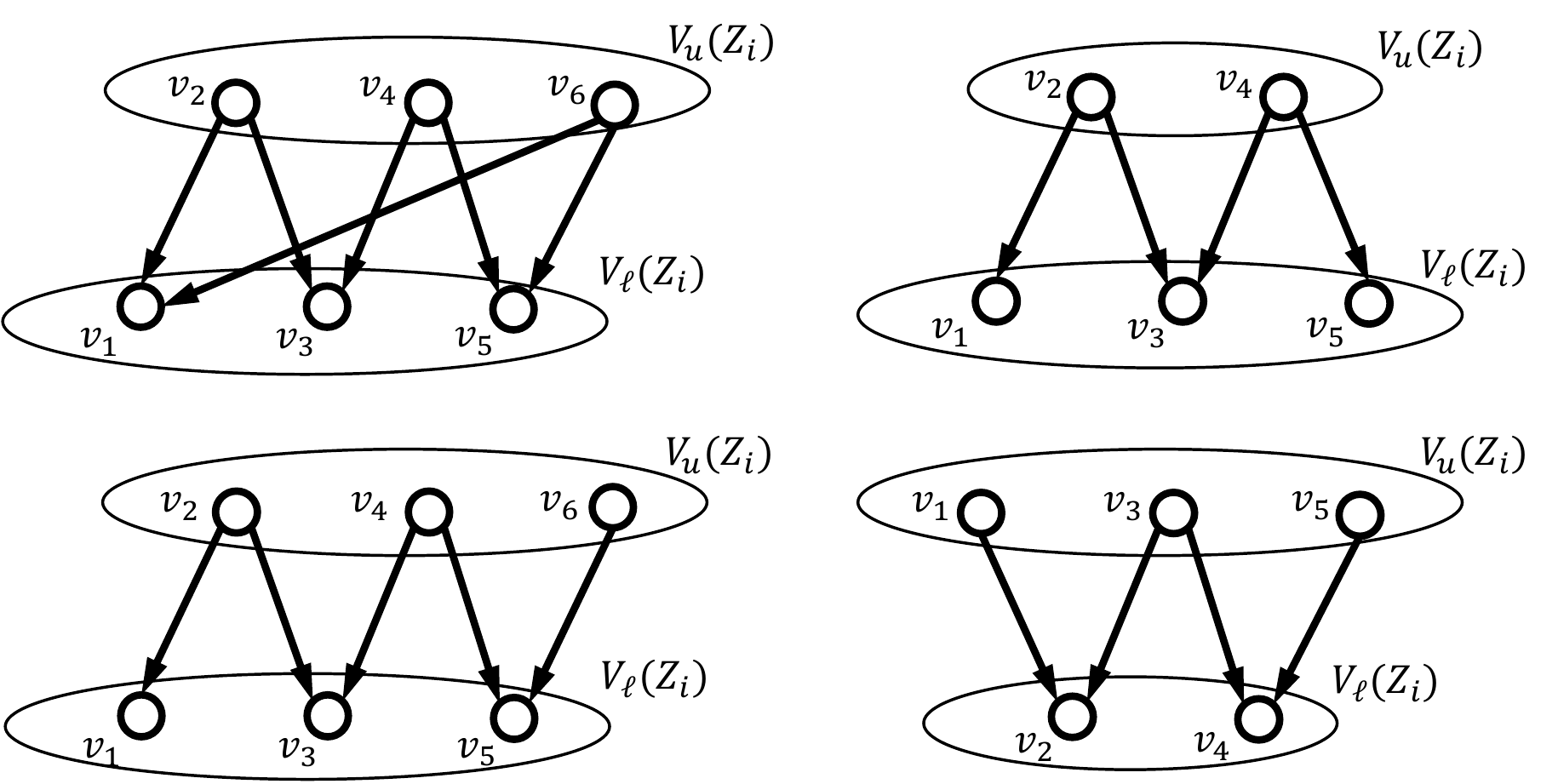}
  \caption{An illustration of Proposition~\ref{four_types_of_zig-zag_trail}: upper and lower vertices in a crown (top left), an M-fence (top right), an N-fence (bottom left) and a W-fence (bottom right).}
  \label{fig:upper, lower}
\end{figure}

\begin{proposition}\label{number_of_U/L_vertices_in_MCST}
   For a rooted almost-binary phylogenetic $X$-network $N$, let $T$ be a maximum covering subtree of $N$. Then, for any maximal zig-zag trail $Z$, $|V_u(Z)\cap V(T)| \leq |V_\ell(Z) \cap V(T)|$ holds.
\end{proposition}

\begin{proof}
  On the contrary, suppose that $|V_u(Z)\cap V(T)| > |V_\ell(Z) \cap V(T)|$ holds. 
  Let $G$ be a subgraph $G :=(V(Z)\cap V(T) , E(Z)\cap E(T))$ of $Z$, which is a forest (i.e.  a collection of trees). By a slight abuse of notation, let  $V_\ell(G):=V_\ell(Z)\cap V(T)$ and $V_u(G):=V_u(Z)\cap V(T)$, respectively. 
Since we have $outdeg_G(v) \geq 1$ for any $v\in V_u(G)$,  $\sum_{v\in V_u(G)} outdeg_G(v) \geq \sum_{v\in V_u(G)} 1 = |V_u(G)|$ holds. By the assumption,  $|V_u(G)| > |V_\ell(G)|$.
 By the hand-shaking lemma, $\sum_{v\in V_u(G)} outdeg_G(v) = \sum_{v\in V_\ell(G)} indeg_G(v)$. Then we have $\sum_{v\in V_\ell(G)} indeg_G(v) >|V_\ell (G)|$, which  means that there exists a vertex $v \in V_\ell (G)$ such that $indeg_G(v) \geq 2$, which contradicts that $G$ is a forest. 
  This completes the proof.
\end{proof}

Proposition~\ref{number_of_U/L_vertices_in_MCST} yields the following useful result.




\begin{lemma}\label{upper_vertex_and_missing_vertex}
   Let $Z$ be a maximal W-fence in a rooted almost-binary phylogenetic $X$-network $N$ and let $T$ be a maximum covering subtree of $N$.  Then, there exists at least one upper vertex $v\in V_u(Z)$ that is not included in $T$. 
\end{lemma}

\begin{proof}
 Since $Z$ is a W-fence, the equation~\ref{eq:diff} in Proposition~\ref{four_types_of_zig-zag_trail} gives $|V_\ell(Z)| = |V_u(Z)| -1$. Proposition~\ref{number_of_U/L_vertices_in_MCST} gives  $|V_u(Z)\cap V(T)| \leq |V_\ell(Z) \cap V(T) |$. Clearly, $|V_\ell(Z)\cap V(T)| \leq |V_\ell(Z)|$. Thus, $|V_u(Z) \cap V(T)| \leq |V_u(Z)| - 1$, which means that $T$ does not include at least one upper vertex of $Z$. This completes the proof. 
\end{proof}

\begin{theorem}\label{infimum_of_eta}
  For any rooted almost-binary phylogenetic $X$-network $N$, $\eta^\ast(N) \geq |\mathcal{W}_N|$ holds.
\end{theorem}

\begin{proof}
 Applying Lemma~\ref{upper_vertex_and_missing_vertex} to each maximal W-fence $Z$ in $N$, we obtain $\eta^\ast(N) \geq |\mathcal{W}_N|$. 
\end{proof}

\subsection{Necessary and Sufficient Condition for when $\eta^\ast(N)=|\mathcal{W}_N|$ holds}\label{subsec:eq}
In this section, we will give a characterisation of rooted almost-binary phylogenetic $X$-networks $N$ with $\eta^\ast(N)=|\mathcal{W}_N|$ and thus show that the inequality in Theorem~\ref{infimum_of_eta} gives a tight lower bound for $\eta^\ast (N)$. We now give the necessary definitions.

 \begin{definition}\label{def:above-below}
  Let $\mathcal Z = \{Z_1, \dots, Z_k\}$ be the maximal zig-zag trail decomposition of a rooted almost-binary phylogenetic $X$-network $N$.  
  For any distinct $Z_i, Z_j\in \mathcal Z$, we say that $Z_i$ is \emph{above} $Z_j$ or $Z_j$ is \emph{below} $Z_i$ if there exists a vertex  $v\in V_\ell (Z_i) \cap V_u(Z_j)$.  
An \emph{M-W pair} of $N$ is defined to be an ordered pair $(Z_i, Z_j)$ such that $Z_i\in \mathcal{M}_N$, $Z_j\in \mathcal{W}_N$, and $Z_i$ is above $Z_j$. 
 \end{definition}
 
%
%

With the `above-below' relation between maximal zig-zag trails in $N$ introduced in Definition~\ref{def:above-below}, we can formally describe whether and how any pair of maximal zig-zag trails share a vertex. As defined in Definition~\ref{def:above-below} and as Proposition~\ref{the_spread_of_missing_vertices} states, particularly important for MCSP is an M-W pair. 

\begin{proposition}\label{the_spread_of_missing_vertices}
Let $T$ be a maximum covering subtree of a rooted almost-binary $X$-network $N$. Suppose $Z$ is a maximal W-fence in $N$ that is not below any maximal M-fence in $N$. Then there exists a maximal zig-zag trail $Z^\prime$ (not an M-fence) above $Z$ such that there exists an upper vertex $v\in V_u(Z^\prime)$ with $v\not\in V(T)$.
\end{proposition}
  
  \begin{proof}
We  prove that there exists a maximal zig-zag trail $Z^\prime$ above $Z$ such that $Z^\prime$ is not an M-fence and $|V_{u}(Z^\prime)\cap V(T)| \leq |V_{u}(Z^\prime)|-1$ holds. 
  By Lemma~\ref{upper_vertex_and_missing_vertex}, there exists an upper vertex $u$ of $Z$ such that $u  \not\in V(T)$. Let $\tilde{Z}$ be a maximal zig-zag trail in $N$ with $u \in V_\ell(\tilde{Z})$. By Definition~\ref{upper,lower}, $\tilde{Z}$ is above $Z$. Then, by $u  \not\in V(T)$, $|V_\ell(\tilde{Z})\cap V(T)| \leq |V_\ell(\tilde{Z})|-1$ holds. As $\tilde{Z}$ is not an M-fence, by equation~(\ref{eq:diff}) in Proposition~\ref{four_types_of_zig-zag_trail}, we have $|V_\ell(\tilde{Z})| \leq |V_{u}(\tilde{Z})|$. Therefore, we can conclude that $|V_\ell(\tilde{Z})\cap V(T)|\leq |V_{u}(\tilde{Z})|-1$. As Proposition~\ref{number_of_U/L_vertices_in_MCST} implies $|V_u(\tilde{Z})\cap V(T)|\leq |V_\ell(\tilde{Z})\cap V(T)|$, we obtain $|V_{u}(\tilde{Z})\cap V(T)|\leq  |V_{u}(\tilde{Z})|-1$. Thus, $\tilde{Z}$ has the desired property. This completes the proof.
\end{proof}


\begin{definition}\label{MWmatching}
Let $N$ be a rooted almost-binary phylogenetic $X$-network that has at least one maximal W-fence, and 
let $\mathcal{M}_N$ and $\mathcal{W}_N$ be the sets of maximal M-fences and maximal W-fences in $N$, respectively. 
An \emph{M-W matching} in $N$ is defined to be a non-empty set $S=\{(M_1, W_1), \dots, (M_p, W_p)\}$  of M-W pairs in $N$ such that $M_i \neq M_j$ and $W_i \neq W_j$ hold for any $i\neq j \in [1, p]$.
An M-W matching is said to be \emph{$\mathcal{W}_N$-saturated} if $|S|=|\mathcal{W}_N|$.
\end{definition}


See the left of Figure~\ref{fig:W-saturated} for an illustration of an M-W pair $(Z_1, Z_2)$.  See also the right of Figure~\ref{fig:W-saturated} for an illustration of a $\mathcal{W}_N$-saturated matching $\{(Z_1, Z_3), (Z_2, Z_4)\}$. We note that there exists no $\mathcal{M}_N$-saturated M-W matching for any $N$, as in Remark~\ref{rem:MW}.

  \begin{figure}[hbt]
    \centering
    \includegraphics[scale=0.5]{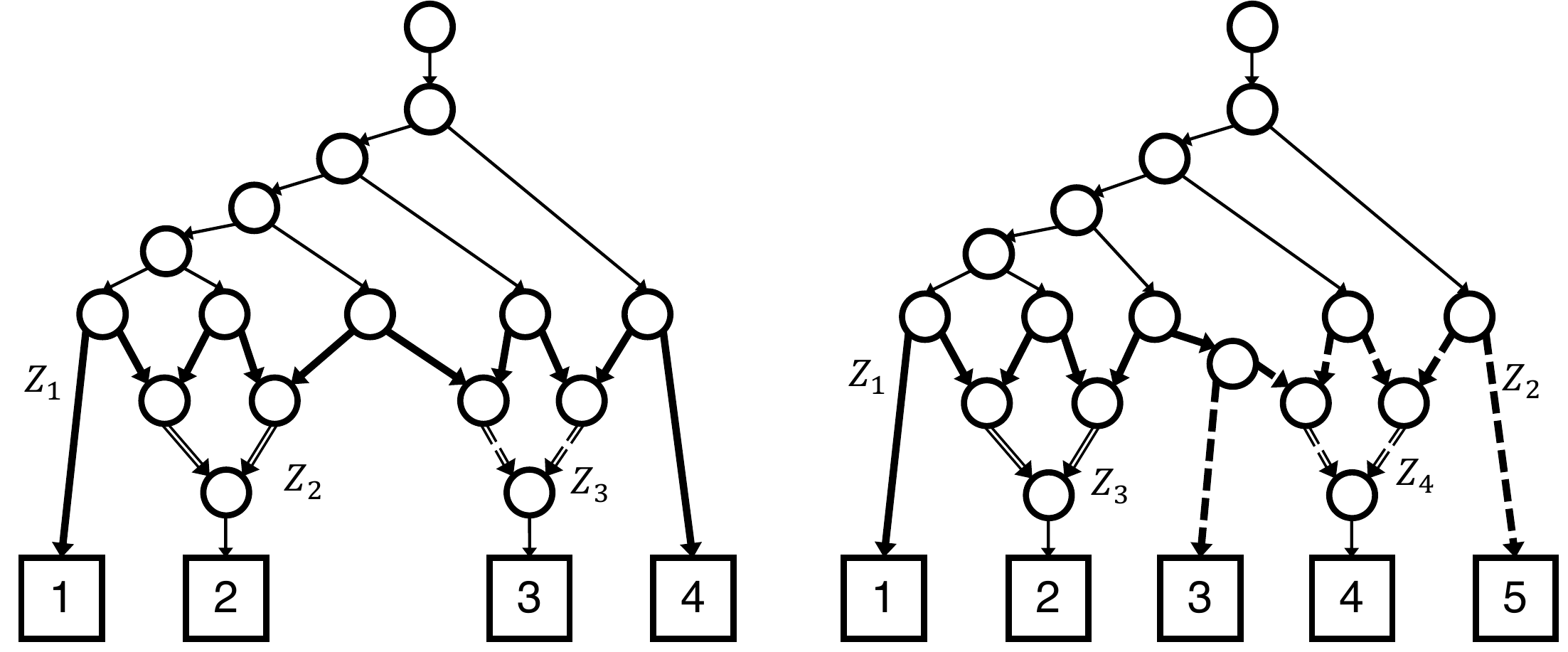}
    \caption{
Left: a rooted almost-binary phylogenetic $X$-network $N$ where there exists no $\mathcal{W}_N$-saturated M-W matching. Right: a rooted almost-binary phylogenetic $X$-network $N$ that has a $\mathcal{W}_N$-saturated M-W matching.
  \label{fig:W-saturated}
 }
  \end{figure}

\begin{figure}[hbt]
  \centering
  \includegraphics[scale=0.5]{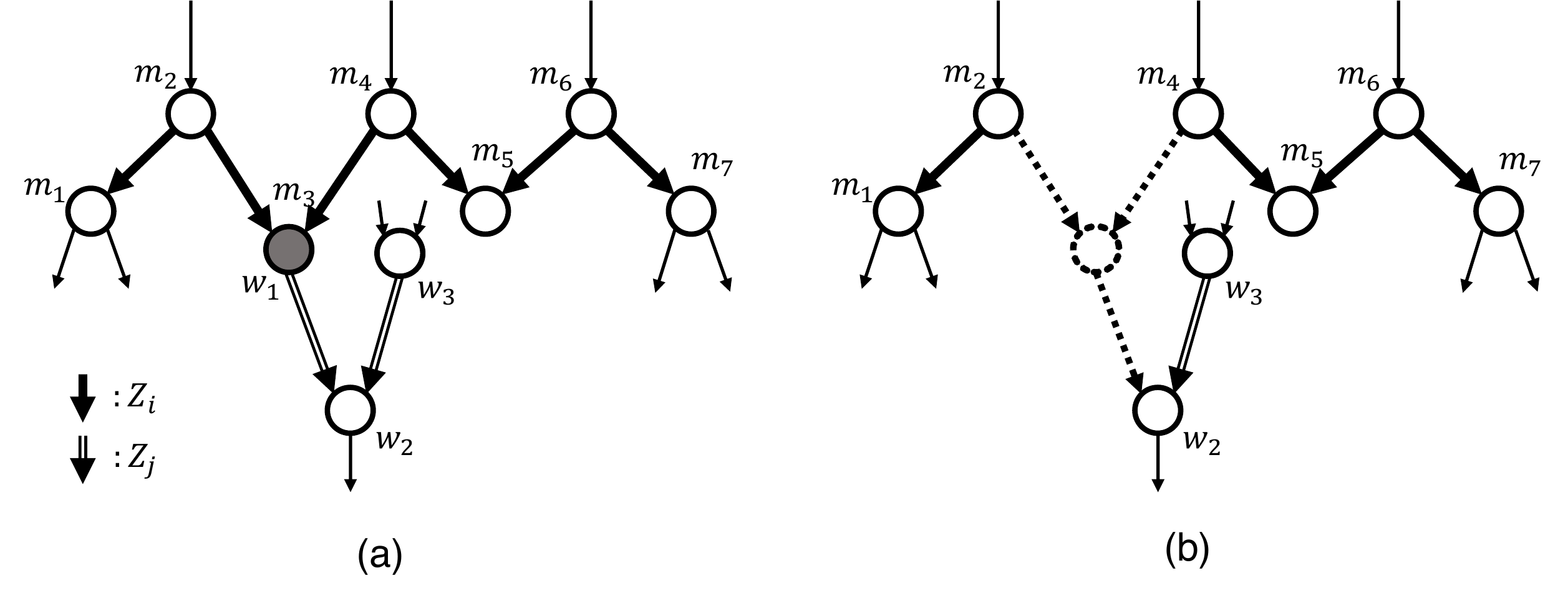}
  \caption{
(a): An M-W pair $(Z_i, Z_j)$ in a rooted almost-binary phylogenetic $X$-network, where the vertices of the maximal M-fence $Z_i$ and those of the maximal W-fence $Z_j$ are indicated using $m_i$ and $w_i$, respectively.  Resolving $(Z_i, Z_j)$ at vertex $m_3 = w_1$ results in the graph depicted in (b).
  \label{fig:MWpair}
  }
\end{figure}

\begin{remark}~\label{rem:MW}
The number $|\mathcal{M}_N|$ of M-fences is greater than or equal to the number $|\mathcal{W}_N|$ of W-fences. 
\end{remark}
\begin{proof}
The proof will be completed if we can show that $|\mathcal{M}_N| - |\mathcal{W}_N| = |X| - 1$ for any $N$ with $|X|\geq 2$. For any non-leaf vertex $v$ of $N$, there exists a maximal zig-zag trail $Z$ of $N$ such that $v$ is an upper vertex of $Z$.  
Therefore, the number of vertices of $N$ that are contained in $V_u (Z)$ of some maximal zig-zag trail $Z$ of $N$ equals $|V(N)|-|X|$. 
On the other hand, for any non-root vertex $v$ of $N$, there exists a maximal zig-zag trail $Z$ of $N$ such that $v$ is a lower vertex of $Z$, and thus the number  of vertices of $N$ that are contained in $V_\ell (Z)$ of some maximal zig-zag trail $Z$ of $N$ equals $|V(N)|-1$.
Then, we have 
$\sum_{Z\in \mathcal{Z}} (|V_{\ell}(Z)| -|V_u(Z)|) = |X|-1$. By the equation~(\ref{eq:diff}) in  Proposition~\ref{four_types_of_zig-zag_trail}, $|X|-1=1\times |\mathcal{M}_N| + 0 \times (|\mathcal{Z}|-|\mathcal{M}_N|-|\mathcal{W}_N|)  +  (-1)\times |\mathcal{W}_N|$. 
Thus, $|\mathcal{M}_N|-|\mathcal{W}_N| = |X|-1$, which completes the proof.
\end{proof}




To consider MCSP, we define the operation of eliminating an M-W pair in $N$ as follows. The resulting network $N^\prime$ is still a rooted almost-binary phylogenetic $X$-network, as ensured by Lemma~\ref{change}.

\begin{definition}\label{def:resolution}
  For an M-W pair $(Z_i, Z_j)$ in a rooted almost-binary phylogenetic $X$-network $N$ and any vertex $v\in V_\ell(Z_i)\cap V_u(Z_j)$, we say \emph{resolving $(Z_i, Z_j)$ at $v$} to mean the operation of creating the subgraph $N -\{v\}$ of $N$ (as illustrated in Figure~\ref{fig:MWpair}b).
\end{definition}




\begin{lemma}\label{change}
  Let $N$ be a rooted almost-binary phylogenetic $X$-network that contains an M-W pair $(Z_i, Z_j)$. Let $N^\prime$ be the graph obtained by resolving $(Z_i, Z_j)$ at any vertex $v$ in $V_\ell(Z_i)\cap V_u(Z_j)$. Then, $N^\prime$ is also a rooted almost-binary phylogenetic $X$-network. Moreover, $N^\prime$ satisfies $|\mathcal{W}_{N^\prime}| = |\mathcal{W}_{N}| -1$ and $|\mathcal{M}_{N^\prime}|=|\mathcal{M}_{N}|-1$.
\end{lemma}

\begin{proof}
Suppose $(Z_i, Z_j)$ is an M-W pair with $Z_i :  m_1 < m_2 > m_3 < \cdots > m_{2p-1} < m_{2p} > m_{2p+1}$ and $Z_j : w_1 > w_2 < w_3 > \cdots < w_{2q-1} > w_{2q} < w_{2q+1}$ ($p, q \in \mathbb{N}$). It is not hard to see that $N^\prime$ is still an almost-binary phylogenetic $X$-network. 


 Then, Theorem~\ref{uniquely.decomposable} ensures that there exists a unique maximal zig-zag trail decomposition, $\mathcal Z$ and $\mathcal Z^\prime$ ,  of $N$ and $N^\prime$, respectively. 
We now examine the following cases for $v\in V_\ell(Z_i)\cap V_u(Z_j)$ and will prove that both $|\mathcal{W}_{N^\prime}| = |\mathcal{W}_{N}| -1$ and $|\mathcal{M}_{N^\prime}|=|\mathcal{M}_{N}|-1$ holds in each case.  


Case 1: when $v$ is a vertex with $indeg_N(v)=1$ and $outdeg_N(v)=2$ (see also Figure~\ref{fig:case1}).  In this case, the edge incoming to $v$ is either $(m_2, m_1)$ or $(m_{2p}, m_{2p+1})$; without loss of generality, we may assume   $v=m_1$. We can also see that $v=w_k$ ($k\neq 1, 2q+1$). Then, the resolution of $(Z_i, Z_j)$ at $v$ deletes the three edges of $N$, that is, $(m_2, m_1), (w_{2k+1}, w_{2k}), (w_{2k+1}, w_{2k+2})$.
Let $E^\prime:=E(Z_i)\setminus \{(m_2, m_1)\}$ and $E^{\prime\prime}:=E(Z_j)\setminus \{(w_{2k+1}, w_{2k}), (w_{2k+1}, w_{2k+2})\}$, and let $G[E^\prime]$ and $G[E^{\prime\prime}]$ be the subgraphs of $N^\prime$ induced by $E^\prime$ and by $E^{\prime\prime}$, respectively. Since $(m_2, m_1)$ is a terminal edge of $Z_i$, we easily see that $G[E^\prime]$ is a maximal zig-zag trail in $N^\prime$. Moreover,  $G[E^\prime]$ is an N-fence since $G[E^\prime]$ has an odd number of edges. 
Regarding  the consecutive edges $(w_{2k+1}, w_{2k})$ and $(w_{2k+1}, w_{2k+2})$ of $Z_j$,  $G[E^{\prime\prime}]$ consists of two maximal N-fences of $N^\prime$ since neither of the two edges is a terminal edge of $Z_j$. 
Thus, resolving $(Z_i, Z_j)$ at $v$ turns the maximal M-fence $Z_i$ of $N$ into a maximal N-fence of $N^\prime$ and the maximal W-fence $Z_j$ of $N$ into two maximal N-fences of $N^\prime$.  Hence,  $|\mathcal{W}_{N^\prime}| = |\mathcal{W}_{N}| -1$ and $|\mathcal{M}_{N^\prime}|=|\mathcal{M}_{N}|-1$.

Case 2: when $v$ is a vertex with $indeg_N(v)=2$ and $outdeg_N(v)=1$ (see also Figure~\ref{fig:case2}).  In this case, the edge outgoing from $v$ is either $(w_1, w_2)$ or $(w_{2q+1}, w_{2q})$; without loss of generality, we may assume   $v=w_1$. Then, by an argument similar to Case 1, we obtain $|\mathcal{W}_{N^\prime}| = |\mathcal{W}_{N}| -1$ and $|\mathcal{M}_{N^\prime}|=|\mathcal{M}_{N}|-1$.


%

Case 3: when $v$ is a vertex with $indeg(v) = outdeg(v) = 1$ (see also Figure~\ref{fig:case3}). 
In this case, the edge incoming to $v$ (resp. the edge outgoing from $v$) is a terminal edge of $Z_i$ (resp. $Z_j$). Similarly to the above, we can see that resolving $(Z_i, Z_j)$ turns the maximal M-fence $Z_i$ (resp. $Z_j$) into a maximal N-fence of $N^\prime$. Thus, we obtain $|\mathcal{W}_{N^\prime}| = |\mathcal{W}_{N}| -1$ and $|\mathcal{M}_{N^\prime}|=|\mathcal{M}_{N}|-1$.

%

Case 4: when $v$ is a vertex with $indeg(v) = outdeg(v) = 2$ (see also Figure~\ref{fig:case4}). 
In this case, the two edges incoming to $v$ (resp. the two edge outgoing from $v$) are non-terminal edges of $Z_i$ (resp. $Z_j$). Similarly to the above, we can see that resolving $(Z_i, Z_j)$ breaks the maximal M-fence $Z_i$ (resp. $Z_j$) into two  maximal N-fences of $N^\prime$. Thus, we obtain $|\mathcal{W}_{N^\prime}| = |\mathcal{W}_{N}| -1$ and $|\mathcal{M}_{N^\prime}|=|\mathcal{M}_{N}|-1$. This completes the proof.

%
%
%
%
%

%
\end{proof}

\begin{figure}[hbt]
  \centering
  \includegraphics[scale=0.5]{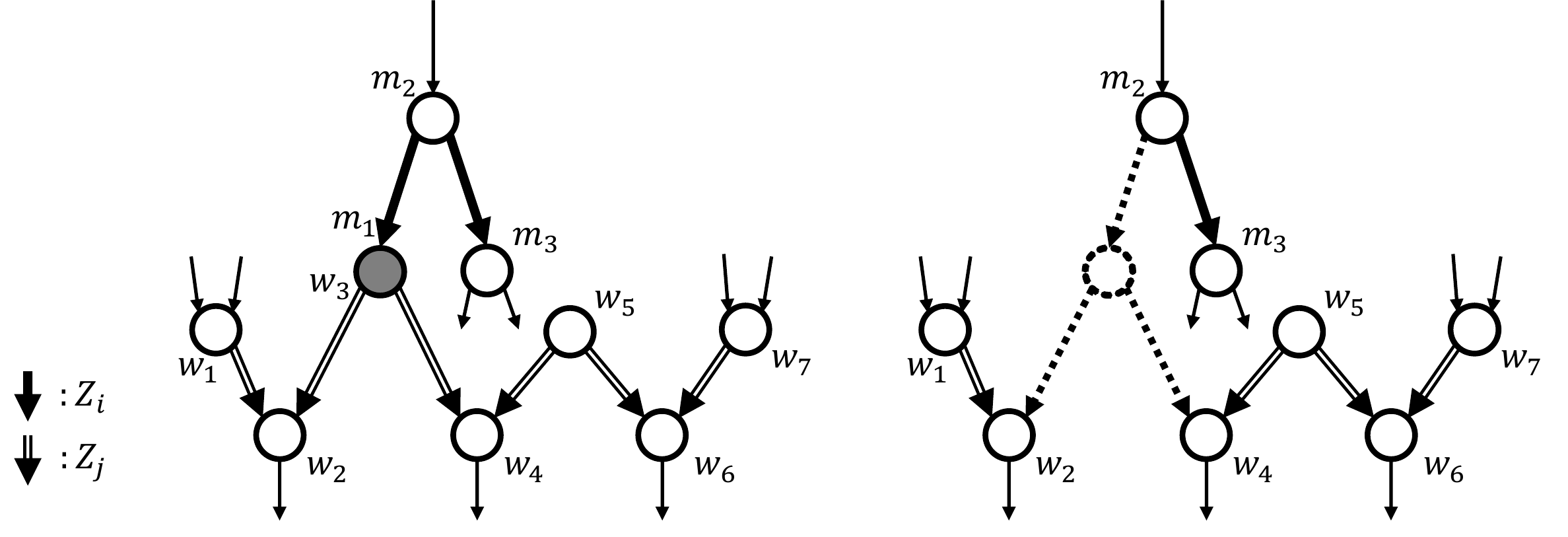}
  \caption{Case 1 in the proof of Lemma~\ref{change} ($p=1, q= 3$). 
  \label{fig:case1}
  }
\end{figure}

\begin{figure}[hbt]
  \centering
  \includegraphics[scale=0.5]{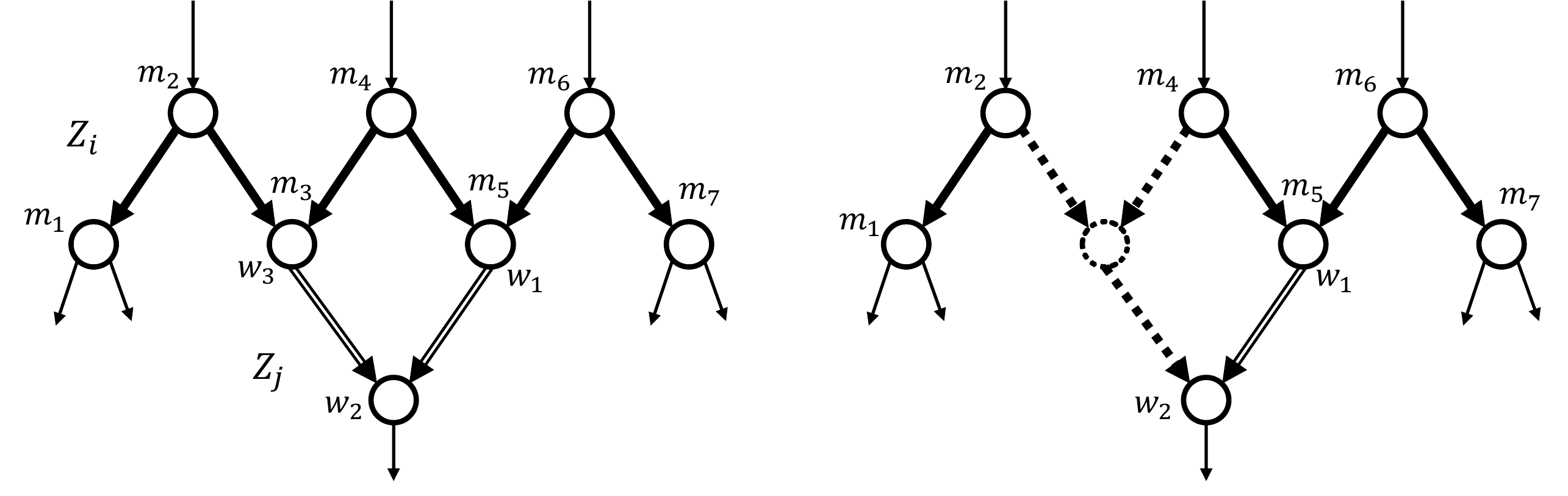}
  \caption{Case 2 in the proof of Lemma~\ref{change}.
\label{fig:case2}
 }
  \end{figure}

\begin{figure}[hbt]
  \centering
  \includegraphics[scale=0.5]{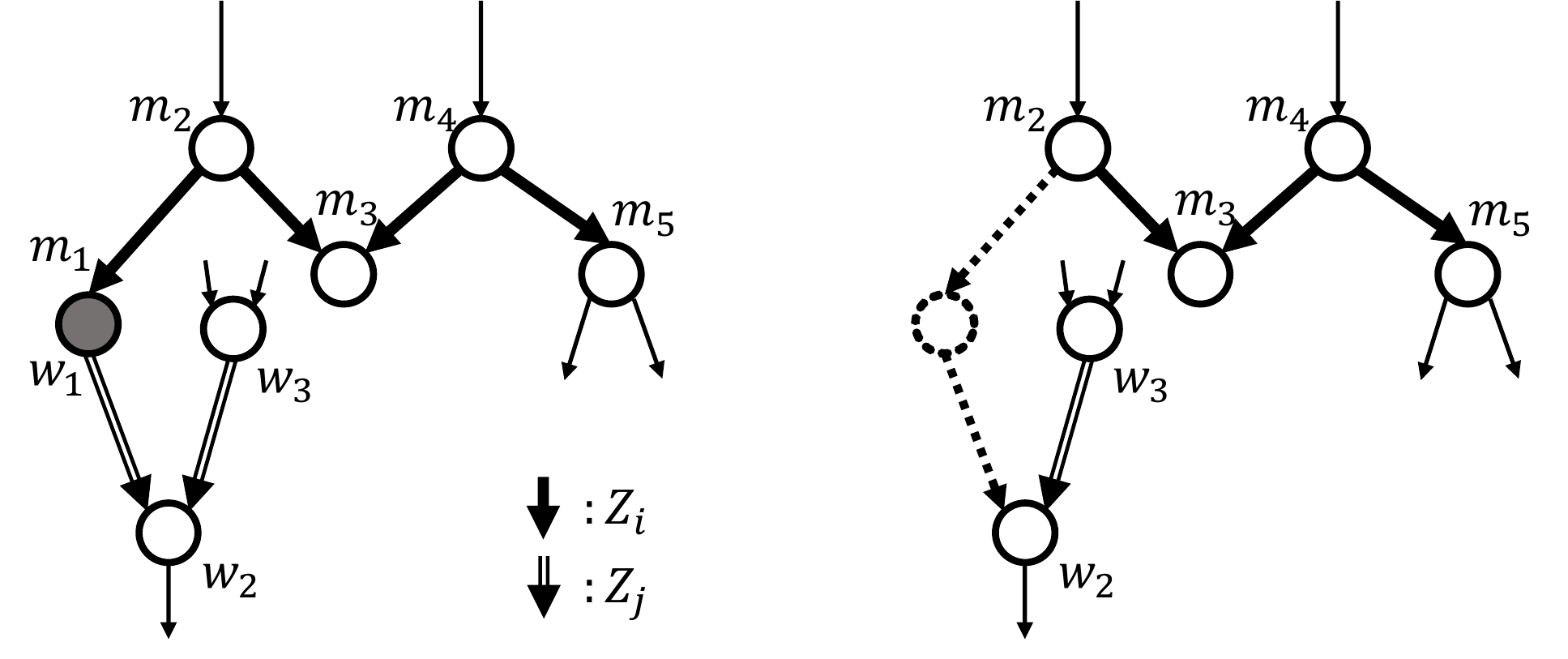}
  \caption{Case 3 in the proof of Lemma~\ref{change}.
\label{fig:case3}
  }
\end{figure}

\begin{figure}[hbt]
  \centering
  \includegraphics[scale=0.5]{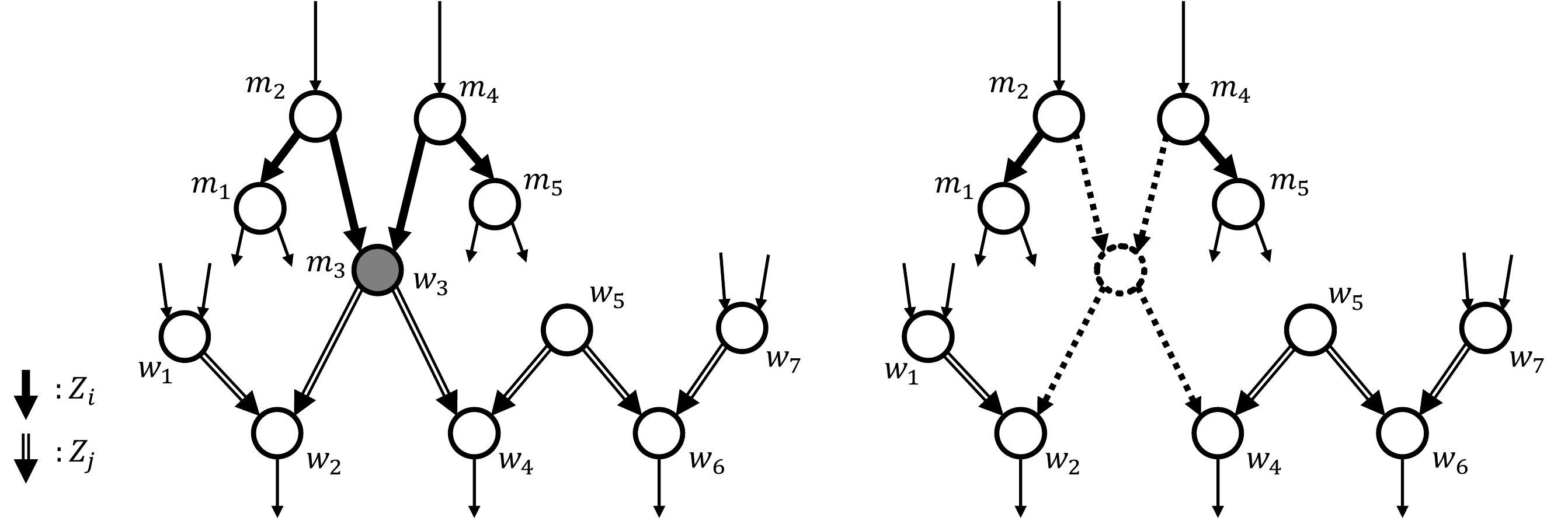}
  \caption{Case 4 in the proof of Lemma~\ref{change}.
 \label{fig:case4}
  }
\end{figure}

We now prove a special case of the main result.

\begin{proposition}\label{specialcase}
In the case of $|\mathcal{W}_N| =  1$, $\eta^\ast(N) = |\mathcal{W}_N|=1$ holds if and only if there exists an $\mathcal{W}_N$-saturated M-W matching of $N$. 
\end{proposition}
\begin{proof}
Let $Z_j$ be the only maximal W-fence in $N$. 
We first show that if there exists exactly one $\mathcal{W}_N$-saturated M-W matching of $N$, then $\eta^\ast(N) = |\mathcal{W}_N|=1$ holds. 
Let  $(Z_i, Z_j)$ be the (only) M-W pair in $N$,  and let $N^\prime$ be the subgraph of $N$ obtained by resolving $(Z_i, Z_j)$ at a vertex $v\in V_\ell(Z_i)\cap V_u(Z_j)$. Lemma \ref{change} says that $N^\prime$ is a rooted almost-binary phylogenetic $X$-network and that $N^\prime$ has no maximal W-fence. Then, $N^\prime$ has a subdivision tree $T^\prime$ by Proposition \ref{spanning}. Since $N^\prime:=N-\{v\}$ is a subgraph of $N$,  there exists a subtree $T$ of $N$ such that $V(N)\setminus V(T)=\{v\}$ and  $T$ is isomorphic to $T^\prime$. 
As $T$ shares the same root and leaf-set $X$ with $N$, $T$ is a covering subtree of $N$ with $|V(N)|-|V(T)|=1$.  As Theorem \ref{infimum_of_eta} gives $1 = |\mathcal{W}_N| \leq \eta^\ast(N)$, we see that $T$ is a maximum covering subtree of $N$. 
Thus, $\eta^\ast(N) = 1$ holds.

To prove the converse, assume that there exists no $\mathcal{W}_N$-saturated M-W matching of $N$. This means that $N$ contains no M-W pair $(Z_i, Z_j)$ for the maximal W-fence $Z_j$. By Lemma~\ref{upper_vertex_and_missing_vertex}, there exists a vertex $v\in V_u(Z_j)$ that is not included in a maximum covering subtree $T$ in $N$. Then, by Proposition~\ref{the_spread_of_missing_vertices}, there exists a maximal zig-zag trail $Z$ above $Z_j$ with $v\in V_\ell(Z)$, where $Z$ is not an M-fence. Proposition~\ref{the_spread_of_missing_vertices} also states that there exists $u\in V_{u}(Z)$ with $u\not \in V(T)$. Thus, $T$ contains neither $u$ nor $v$. Hence, we obtain $\eta^\ast(N) \geq 2 > 1$ as desired.

\end{proof}

Finally, we give a characterisation of rooted almost-binary phylogenetic $X$-networks $N$ with  $\eta^\ast(N) = |\mathcal{W}_N|$ as follows.

\begin{theorem}\label{equality}
  For a rooted almost-binary phylogenetic $X$-network $N$,  $\eta^\ast(N) = |\mathcal{W}_N|$ holds if and only if there is no maximal W-fence or there exists an $\mathcal{W}_N$-saturated M-W matching of $N$.
\end{theorem}
\begin{proof}
As the case of $|\mathcal{W}_N| = 0$ is trivial, we may assume $|\mathcal{W}_N| \geq 1$. The goal is to show that $\eta^\ast(N) = |\mathcal{W}_N|$ holds if and only if there exists an $\mathcal{W}_N$-saturated M-W matching of $N$. 
The proof is by induction on the number $|\mathcal{W}_N|$ of maximal W-fences of $N$. In the case of $|\mathcal{W}_N| = 0$, the assertion holds by Proposition~\ref{spanning}. When $|\mathcal{W}_N| = 1$, the assertion holds by Proposition~\ref{specialcase}. 

Assuming that Theorem~\ref{equality} holds for any $N^\prime$ with $|\mathcal{W}_{N^\prime}|\leq i$, we will show that the assertion holds for any $N$ with  $|\mathcal{W}_N| = i + 1$.
%
%
Suppose there exists a $\mathcal{W}_N$-saturated M-W matching of $N$.  Let $(Z_i, Z_j)$ be an element of $\mathcal{W}_N$-saturated M-W matching of $N$ and let $N^\prime$ be an almost-binary phylogenetic $X$-network  with $|\mathcal{W}_{N^\prime}| = i$ obtained by the resolution of $(Z_i, Z_j)$  (Lemma~\ref{change} is used here). Then, by the induction hypothesis, we have $\eta^\ast  (N ^\prime ) = i$. Let $T^\prime$ be a maximum covering subtree $T^\prime$ of $N^\prime$.  Similarly to the argument in the proof of Proposition~\ref{specialcase}, $N$ contains a covering subtree $T$ that is isomorphic to $T^\prime$.  As $V(N)\setminus V(N^\prime)=\{v\}$, the number of vertices of $N$ that are not covered by $T$ equals $i + 1$. As Theorem \ref{infimum_of_eta} guarantees
$i+1 = |\mathcal{W}_N| \leq \eta^\ast(N)$, we can conclude that $T$ is a maximum covering subtree of $N$ and $\eta^\ast (N)=i+1$ holds. 



To prove the converse, assuming that there is no $\mathcal{W}_N$-saturated M-W matching of $N$, we will show that  $\eta^\ast(N) >|\mathcal{W}_N| (= i+1)$ holds. Let $Z_j$ be a maximal W-fence that is not in any M-W pair in $N$.  Then, by Lemma~\ref{upper_vertex_and_missing_vertex}, there exists an upper vertex $v$ of $Z_j$ with $v\not \in V(T)$ for any maximum covering subtree  $T$ of $N$.
 Let $Z$ be the maximal zig-zag trail in $N$ that is above $Z_j$ and has $v$ as its lower vertex.
When $Z$ is not an M-fence, Proposition~\ref{the_spread_of_missing_vertices} says that there exists a vertex $u \in V_{u} (Z)$ with $u \not \in V(T)$, which implies $\eta^\ast (N)\geq i+2$. Next, we must consider the case when $Z$ is an M-fence.  In fact, even if there is no $\mathcal{W}_N$-saturated M-W matching, it is possible that there exists a maximal M-fence above each maximal W-fence. This is because a maximal M-fence can be above two or more maximal W-fences (see the left panel of Figure~\ref{fig:W-saturated} for an example). 
Let $Z$ be a maximal M-fence that is above two distinct maximal W-fences $Z_j$ and $Z_j^\prime$ in $N$ and let $k$ be the number of upper vertices of $Z$. According to Proposition~\ref{four_types_of_zig-zag_trail}, the number of lower vertices of $Z$ is given by $k+1$. Lemma~\ref{upper_vertex_and_missing_vertex} ensures the existence of two distinct vertices $p \in V_{u}(Z_j)$ and $q \in V_{u}(Z_j^\prime)$ such that neither $p$ nor $q$ is in $V(T)$. Since $p, q \in V_{\ell}(Z)$, we have $|V_\ell(Z) \cap V(T)| \leq k+1-2 = k-1$.
Proposition~\ref{number_of_U/L_vertices_in_MCST} yields $|V_u(Z)\cap V(T)| \leq |V_\ell(Z) \cap V(T)|$, leading to $|V_u(Z)\cap V(T)|\leq k-1$. Thus, at least one upper vertex of $Z$ is not in $V(T)$. Consequently, we deduce that $\eta^\ast (N)\geq i+2$.
This completes the proof. 
\end{proof}




\section{Conclusion}
In this paper, we have considered the number $\eta^\ast(N)$ of uncovered vertices in a phylogenetic network $N$ by a new approach using a canonical decomposition of phylogenetic networks called  the maximal zig-zag trail decomposition. We defined an above/below relationship between maximal zig-zag trails by defining the concept of upper and lower vertices of each maximal zig-zag trail, and then introduced the new concept of M-W pairs and $\mathcal{W}_N$-saturated M-W matchings. We gave a tight lower bound  $|\mathcal{W}_N|\leq \eta^\ast(N)$ and proved that $|\mathcal{W}_N|=\eta^\ast (N)$ holds if and only if there exists a $\mathcal{W}_N$-saturated M-W matching in $N$. 

The results in this paper have many different implications. First, as $|\mathcal{W}_N|$ can be computed in linear time \cite{structure_theorem}, Theorem~\ref{equality} means that MCSP can be solved in linear time for (binary or) almost-binary phylogenetic networks with $\mathcal{W}_N$-saturated M-W matching. As the algorithm in \cite{Maximum_Covering} requires cubic time when $N$ is binary, our decomposition-based approach is faster when $N$ is binary or almost-binary. Second, since tree-based networks $N$ are a trivial instance of MCSP (because $|\mathcal{W}_N|=\eta^\ast (N)=0$), phylogenetic networks with $\mathcal{W}_N$-saturated M-W matchings can be understood as a generalisation of tree-based networks. Recalling $\delta^\ast(N)=|\mathcal{W}_N|$~\cite{structure_theorem}, we have thus  clarified the relationship between the two indices $\delta^\ast(N)$ and $\eta^\ast (N)$ for measuring the deviation of $N$ from being tree-based, which were originally proposed by Francis--Semple--Steel  \cite{New_characterisations}. For future work, it would be interesting to investigate the mathematical and computational properties of phylogenetic networks with or without $\mathcal{W}_N$-saturated M-W matching.



\bibliography{takatora.bib}

\end{document}